\documentclass{birkjour}

\usepackage[T1]{fontenc}
\usepackage[utf8]{inputenc}
\usepackage[english]{babel}
\usepackage{mathtools}
\usepackage{amsfonts}
\usepackage{amsthm}
\usepackage{comment}
\usepackage{amsmath}
\usepackage{amstext,amsfonts,verbatim}
\usepackage{xcolor}
\usepackage{cite}
\usepackage{braket}
\usepackage{amssymb}

\numberwithin{equation}{section}
\theoremstyle{plain}
\newtheorem{theo}{Theorem}[section]
\newtheorem{prop}[theo]{Proposition}
\newtheorem{lemma}[theo]{Lemma}
\newtheorem{coro}[theo]{Corollary}
\newtheorem{defi}{Definition}
\newtheorem{question}{Question}
\theoremstyle{definition}
\newtheorem{rem}[theo]{Remark}

\newcommand{\T}{\mathbb{T}}
\newcommand{\D}{\mathbb{D}}
\newcommand{\Mult}{\operatorname{Mult}}
\newcommand{\Z}{\mathbb{Z}}

\newcommand{\hol}{\mathrm{Hol}}
\newcommand{\C}{\mathbb{C}}

\newcommand{\B}{\mathcal{B}}
\newcommand{\N}{\mathbb{N}}

\newcommand{\charM}{\langle M\rangle}
\newcommand{\charphi}{\langle \varphi\rangle}
\newcommand{\dyad}{\mathcal{D}}
\newcommand{\dist}{\mathrm{dist}}

\begin{document}
\title[Cyclic and singular]{On the Cyclic Behavior of Singular inner functions in Besov and sequence spaces}
\author[Dayan]{Alberto Dayan \thanks{ (corresponding author)}}
\address{Departament de Matemàtiques \newline Universitat Autònoma de Barcelona \newline 08193 Bellaterra (Barcelona), Spain}
\email{alberto.dayan@uab.cat}

\author[Seco]{Daniel Seco }
\address{Departamento de An\'alisis Matem\'atico e IMAULL \newline Universidad de la Laguna \newline  Avenida Astrof\'isico Francisco S\'anchez, s/n.  \newline Facultad de Ciencias, secci\'on: Matem\'aticas, apdo. 456.  \newline 38200 San Crist\'obal de La Laguna \newline
Santa Cruz de Tenerife,  Spain} \email{dsecofor@ull.edu.es}

\begin{abstract}
We show the existence of singular inner functions that are cyclic in some Besov-type spaces of analytic functions over the unit disc. Our sufficient condition is stated only in terms of the modulus of smoothness of the underlying measure. Such singular inner functions are cyclic also in the space $\ell^p_A$ of holomorphic functions with coefficients in $\ell^p$. This can only happen for measures that place no mass on any Beurling-Carleson set.
\end{abstract}

\thanks{The first author was partially supported by the Emmy Noether Program of the
German Research Foundation (DFG Grant 466012782).\\
The second author is funded through grant PID2024-160185NB-I00 by the Generaci\'on de Conocimiento programme and through grant RYC2021-034744-I by the Ram\'on y Cajal programme from Agencia Estatal de Investigaci\'on (Spanish Ministry of Science, Innovation and Universities), as well as by Vicerrectorado de Investigaci\'on y Transferencia de la Universidad de La Laguna.}



\date{\today}

\maketitle

\section{Introduction}

Let $h$ be a holomorphic function over the unit disc $\D$. We say $h$ is \emph{inner} if $|h(z)| \leq 1$ for every point $z \in \D$ and $|h^*(e^{i\theta})| = 1$ for almost every point of $\T$, where $h^*$ denotes the function defined for almost every point of $\T$ by taking radial limits of $h$.
Let $\mu$ be a positive Borel measure on the unit circle $\T$, singular with respect to the Lebesgue measure. Then we can associate to $\mu$ an inner function $S_\mu$ that we call \emph{singular}, defined on $\D$ by taking:
\begin{equation}
\label{eqn:singularinner}
S_\mu(z):=\exp\left\{-\int_0^{2\pi}\frac{e^{i\theta}+z}{e^{i\theta}-z}\,d\mu(\theta)\right\}\qquad z\in\D.
\end{equation}
Inner functions, and singular ones in particular, have played a crucial role in the complex analysis of the last century, due to their relationship with the structure of the Hardy space $H^2$. This is the set of all functions holomorphic over $\D$ with a square summable sequence of Maclaurin coefficients (with the natural Hilbert space norm). Functions in $H^2$ can be decomposed in an essentially unique manner as the product of an inner function and an \emph{outer} function, that is, a holomorphic function $g$ with no zeros in $\D$, such that $\log|g^*(z)|$ has the mean value property on $\T$.  See \cite{garnett} for the basic theory regarding the $H^2$ space. The seminal work of Beurling, \cite{Beurling48}, shows that inner function are canonical generators of all proper invariant subspaces for the shift operator in the Hardy space. In the present article, we discuss whether there exist singular inner functions that are \emph{cyclic} (for the shift operator) in other spaces of analytic functions over the unit disc. Cyclic functions are those contained in no proper invariant subspace, but in our context, we may say that a function $f$ is cyclic in a space $X$, if there exists a sequence of polynomials $\{p_n\}_{n\in \N}$ such that $\|p_n f-1\|_X\rightarrow 0$. In what follows, we will denote by $[f]_X$ the smallest closed subspace that contains $f\in X$ and that is invariant by multiplication by $z$. If the space $X$ is clear from the context, we will drop it in our notation and call such invariant subspace $[f]$.

Our primary goal will be to describe those singular inner functions that are cyclic in $X$ in the case that $X$ is any $\ell^p_A$ space for $p>2$. With this purpose, we define the space
\[
\ell^{p}_A:=\left\{f\in\hol(\D)\,\bigg|\,\|f\|^p_p:=\sum_{k=0}^\infty|\hat f(k)|^p<\infty\right\},
\] 
where $\hat f(k)$ denotes the Maclaurin coefficient of $f$ of order $k$. The $\ell^p_A$ spaces are Banach spaces whenever $p \geq 1$ (and quasi-Banach for $p>0$). Notice that $p=2$ yields the Hardy space $H^2$. Moreover, for $p \leq 2$ it is clear that there cannot be any singular inner functions which are cyclic as far as Beurling's theorem implies that there are no cyclic inner functions in the Hardy space $\ell^2_A$ other than constants.

Our description of cyclic singular inner vectors in $\ell^p_A$ will take advantage of the relation between coefficient spaces and Besov-type spaces. For all $0< p<\infty$, the Besov space $D^p_{p-1}$ is the space of all analytic functions over $\D$ whose norm
\[
\|f\|_{D^p_{p-1}}:=|f(0)|+\left(\int_\D(1-|z|)^{p-1}|f'(z)|^p~dA(z)\right)^\frac{1}{p}
\]
is finite. 
From \cite[Th. 1.1, (ii) and Th. 1.2]{Girela06} and \cite[Th. 7.1]{Pavlovic} we get
\begin{equation}
\label{eqn:emb_ge2}
H^p\subset D^p_{p-1}\subset\ell^p_A,\qquad p\ge 2
\end{equation}
and 
\begin{equation}
\label{eqn:emb_le2}
\ell^p_A\subset D^p_{p-1}\subset H^p,\qquad 0<p\le2,
\end{equation}
where $H^p$ denotes the standard Hardy space of exponent $p$. Notice that, once again, the case $p=2$ corresponds to the Hardy space $H^2$. Moreover, the second inclusion in \eqref{eqn:emb_le2} implies that no singular inner function can be cyclic in $D^p_{p-1}$, if $1\le p\le 2$. For $p=\infty$ the space $D^p_{p-1}$ coincides with the familiar Bloch space $\B$, that is, the space of those analytic functions over $\D$ such that
\begin{equation}
\label{eqn:bloch}
\sup_{z\in\D}|f'(z)|(1-|z|)<\infty.
\end{equation}
Cyclic singular inner functions in the Bloch space were studied in \cite{Anderson91} via regularity conditions on their defining singular measure. Given a positive Borel measure $\nu$ on $\T$, its modulus of continuity is defined as
\[
\delta_\nu(t):=\sup_{|I|\le t}\{\nu(I)\},
\] 
while its modulus of smoothness is
\[
\omega_{\nu}(t):=\sup\{|\nu(I)-\nu(J)|\,|\, I, J\,\text{adjacent intervals,}\, |I|=|J|\le t\}.
\] 
The main result of \cite{Anderson91} shows the existence of a singular inner function $S_\mu$ that is (weak$^*$) cyclic in $\B$; in fact, any singular measure $\mu$ satisfying 
\begin{equation}
\label{eqn:bothmoduli}
\delta_\mu(t)\le 8t\left(2+96^{-1}\log\log\frac{e}{t}\right), \qquad \omega_\mu(t)\le36\frac{t}{\sqrt{\log\frac{e}{t}}}
\end{equation}
is cyclic in $\B$. The hard part of their argument is to show that a singular measure $\mu$ satisfying \emph{both} bounds in \eqref{eqn:bothmoduli} actually exists.  It is worth noticing that the argument in \cite{Anderson91} extends to the Besov spaces $D^p_{p-1}$ for $2<p<\infty$. In particular, any singular inner function satisfying \eqref{eqn:bothmoduli} has a power that is cyclic for $D^p_{p-1}$, even for $2<p<\infty$. Namely, the same techniques in \cite{Anderson91} yield the existence of a singular inner function that is cyclic in $D^p_{p-1}$, and hence in $\ell^p_A$ via the embedding \eqref{eqn:emb_ge2}. Alternatively, \cite{Limani24}, one can embed a Bloch-type space inside $\ell^p_A$, and use once again the same argument in \cite{Anderson91}.\\

This note grew out of the attempt to give a better description of singular inner vectors in $\ell^p_A$, for $p>2$. Indeed, constructing examples of a singular measure $\mu$ satisfying \eqref{eqn:bothmoduli} is highly non trivial, and the construction in \cite{Anderson91} is not explicit. On the other hand, it is considerably easier to construct a singular measure by prescribing only its modulus of smoothness. Our main results reads as follows:
\begin{theo}
\label{theo:main}
Let  $S_\mu$ be a singular inner function such that 
\begin{equation}
\label{eqn:sufficient}
\omega_\mu(t)\leq C~\frac{t}{\sqrt{\log\frac{e}{t}}}\qquad t\in(0, 1)
\end{equation}
for some $C>0$. Then $S_\mu$ is cyclic in $D^p_{p-1}$ (and hence in $\ell^p_A$) for all $p>2$.
\end{theo}
An explicit construction of a singular measure satisfying \eqref{eqn:sufficient} can be found in \cite[Th. II]{Shapiro68}. The proof of Theorem \ref{theo:main} is contained in Section \ref{sec:proof}, and it can be outlined as follows. The reader familiar with the argument in \cite{Anderson91} notices that the extra assumption in $\delta_\mu$ is required to estimate the dilates of $1/|S_\mu|$ \emph{pointwise}, since in that setting one works with the seminorm \eqref{eqn:bloch}. In the setting $D^p_{p-1}$, one can replace a pointwise estimate of $1/|S_\mu|$ with an estimate of its $L^p$-norm on circles approaching $\T$. To obtain such an estimate, we will relate $1/|S_\mu|$ with a dyadic martingale on the unit circle. The description of such martingales, together with the tools that are necessary to estimate the above mentioned $L^p$-averages, are contained in Section \ref{sec:martingales}. As a result, under an even weaker condition than \eqref{eqn:sufficient} one has that a power of $S_\mu$ is cyclic in $D^p_{p-1}$ (Theorem \ref{theo:power}). In order to conclude that $S_\mu$ itself is cyclic in $D^p_{p-1}$, we need to show that \eqref{eqn:sufficient} implies that $S_\mu$ is a multiplier of $D^p_{p-1}$; this is done in Section \ref{sec:multipliers}. On the other hand, Theorem \ref{theo:main} implies that such singular inner function is not a multiplier of $\ell^p_A$ for all $p\in(1, \infty)\setminus\{2\}$; in fact, it is not in $\ell^p_A$ for all $1<p<2$ (see Remark \ref{rem:multipliers}). Sections \ref{sec:outer} and \ref{sec:log} includes some additional remarks and open problems on the relation between cyclicity in $D^p_{p-1}$, outer functions and logarithmic conditions.\\

We also point out that a necessary condition for $S_\mu$ to be cyclic in $\ell^p_A$ comes from their embeddings into some Bergman-type spaces. Let $\ell^{2, \alpha}_A$ be the space of analytic functions such that 
\begin{equation}
\label{eqn:ell2alpha}
\sum_{n=0}^\infty|\hat{f}(n)|^2(1+n)^{\alpha}<\infty.
\end{equation}
For $\alpha=0$, we recover the Hardy space $H^2$, while for negative $\alpha$ we obtain Bergman-type spaces. It was proven in two works by Korenblum \cite{korenblum81} and Roberts \cite{Roberts85} that for any $\alpha<0$ a singular inner function $S_\mu$ is cyclic in $\ell^{2, \alpha}_A$ if and only if $\mu(E)=0$ for any Beurling-Carleson set $E\subset\T$. We recall that a closed set $E\subset\T$ is a \emph{Beurling-Carleson set} if it has null Lebesgue measure and 
\begin{equation}
\label{eqn:bc}
\sum_n|I_n|\log\frac{1}{|I_n|}<\infty,
\end{equation}
where $\T\setminus E$ is the disjoint union of the open intervals $(I_n)_n$ and $|\cdot|$ denotes the Lebesgue measure on $\T$. Such sets play a crucial role in classic function theory: for instance, they are the boundary zero sets of those holomorphic functions over $\D$ which extend to a $C^\infty$ function on $\overline{\D}$. 

Since for all $\varepsilon>0$ and $p>2$ one has
\[
\ell^p_A\subset\ell^{2, (1+\varepsilon)\left(\frac{2}{p}-1\right)}_A
\]
we obtain that for all $p>2$ any cyclic singular inner function in $\ell^p_A$ must satisfy Korenblum's condition:
\begin{theo}
\label{theo:necessity}
 If $S_\mu$ is cyclic in $\ell^p_A$ for some $p>2$, then $\mu(E)=0$ for any Beurling-Carleson set $E\subset\T$.
\end{theo}
In view of the embedding \eqref{eqn:emb_ge2}, the same holds for cyclic singular inner functions in $D^p_{p-1}$:
\begin{coro}
If $S_\mu$ is cyclic in $D^p_{p-1}$ for some $p>2$, then $\mu(E)=0$ for any Beurling-Carleson set $E\subset\T$.
\end{coro}
On the other hand, the proof of the sufficiency of Korenblum's condition for the Bergman-type spaces $\ell^{2,\alpha}_A$ uses, among other tools, that the multiplier algebra of such Bergman spaces coincides isometrically with $H^\infty$ and that the norm of $f z^n$, for $f$ in such spaces, is asymptotically small, as $n\to\infty$. None of these two properties hold for $\ell^p_A$; nonetheless, the question of whether Korenblum's characterization extends to singular inner vectors in $\ell^p_A$ arises naturally:
\begin{question}
Let $\mu$ be a singular, positive Borel measure on $\T$ such that $\mu(E)=0$ for any Beurling-Carleson set $E\subset\T$, and let $p>2$. Is $S_\mu$ cyclic in $\ell^p_A$? Is it cyclic in $D^p_{p-1}$?
\end{question}

\section{Dyadic Martingales and Singular Measures}
\label{sec:martingales}

Let $\mathcal{D}=\bigcup_n\dyad_n$ be the collection of all dyadic intervals on the torus, where 
\[
\mathcal{D}_n:=\left\{\left[\frac{j}{2^{n}}, \frac{j+1}{2^n}\right)\,\bigg|\, j=0, \dots 2^{n}-1\right\}
\]
denote the collection of all dyadic intervals of length $2^{-n}$. A dyadic martingale is a family $M=(M_I)_{I\in\dyad}$ such that the mean value property
\begin{equation}
\label{eqn:avg}
    M_I=\frac{M_{I_1}+M_{I_2}}{2}
\end{equation}
holds for all $I, I_1, I_2$ in $\dyad$ so that $I=I_1\cup I_2$ and $I_1 \cap I_2 = \emptyset$. Alternatively, one can think of a dyadic martingale as a sequence $(M_n)_n$ of functions on $\T$:
\[
M_n(\theta):=M_{I_\theta},
\]
where $I_\theta$ is the unique interval in $\dyad_n$ that contains $\theta$. Given a positive sequence $\beta=(\beta_n)_n$, we say that a dyadic martingale $M$ is $\beta$-smooth if
\[
|M_I-M_J|\leq \beta_n
\]
for any adjacent intervals $I, J$ in $\dyad_n$. This, together with \eqref{eqn:avg}, implies that $(\beta_n)_n$ also bounds the increments of the martingale $M$; more precisely,
\begin{equation}
\label{eqn:bddincr}
|M_n(\theta)-M_{n-1}(\theta)|\le\frac{\beta_n}{2}
\end{equation}
We denote by $\charM_n$ the quantity
\[
\charM_n(\theta):=\left(\sum_{j=1}^n\left|M_j(\theta)-M_{j-1}(\theta)\right|^2\right)^\frac{1}{2}\qquad \theta\in\T.
\]
Hence, if $M$ is $\beta$-smooth, then 
\[ 
\charM_n(\theta)\leq\frac{1}{2}\left(\sum_{j=1}^n\beta_j^2\right)^\frac{1}{2}.
\]
The following Lemma can be interpreted as a concentration inequality for dyadic martingales, given in terms of the quantity
\[
A_n:=\sup_{\theta\in\T}\charM_n(\theta).
\]
This kind of result was first noted in \cite{Chang85}, while the statement below can be extracted from the proof of \cite[Ch. 4, Lemma 7]{Seip04}:
\begin{lemma}
Let $M$ be a dyadic martingale. Then
\[
|\left\{\theta\in\T\,|\, |M_n(\theta)-M_0(\theta)|\ge s\right\}|\le e^{-\frac{s^2}{2A_n^2}}\qquad s\in (0, \infty).
\]
\end{lemma}
In particular, for any positive $\alpha$, 
\begin{equation}
\label{eqn:conc_mart}
\begin{split}
\int_\T e^{\alpha |M_n(\theta)|}\,d\theta&=1+\int_1^\infty\left|\left\{|M_n|\ge \frac{\log(x)}{\alpha}\right\}\right|\,dx\\
&\leq1+\int_1^\infty e^{-\frac{\log(x)^2}{2\alpha^2A_n^2}}\,dx\\
&\leq C_\alpha A_ne^{\frac{\alpha^2 A_n^2}{2}}.
\end{split}
\end{equation}
Let $\mu$ be a positive Borel measure on $\T$. Then the quantities
\[
M^\mu_I:=\frac{\mu(I)}{|I|}\qquad I\in \dyad
\]
define a dyadic martingale. Moreover, one can estimate its characteristic function $\langle M^\mu_n\rangle$ via smoothness properties of $\mu$.
We say that a function $\varphi$ on $(0,1]$ is almost decreasing if there exists a $c>0$ such that $\varphi(x)\ge c~\varphi(y)$ for  any $y >x$. 

\begin{defi}
\label{defi:smooth}
Let $\varphi:(0,1] \rightarrow (0, +\infty)$ be a function. A positive measure $\mu$ on $\T$ is $\varphi$-smooth if $$\omega_{\mu}(\delta)\leq C\delta\varphi(\delta)$$
for some $C>0$.
\end{defi} 
It was shown in \cite{Stein64} that if $\nu$ is $\varphi$-smooth and 
 \[
 \int_0^1\frac{\varphi^2(t)}{t}<\infty,
 \]
 then $\nu$ is absolutely continuous with respect to Lebesgue measure. On the other hand, if $\int_0^1\varphi(t)^2/t=\infty$, then there exists a $\varphi$-smooth singular measure. We refer the reader to \cite{Kahane69}, where one can find how to construct such singular measures using dyadic martingales. Suppose in addition that $\mu$ is $\varphi$-smooth, and that $\varphi$ is a continuous, increasing function such that $\varphi(t)/t^\beta$ is almost decreasing for some $0<\beta<1$. If 
\[
\charphi(s):=\left(\int_s^1\frac{\varphi^2(t)}{t}\,dt\right)^\frac{1}{2},
\]
then the increments of $M^\mu$ are controlled by the sequence $\beta:=\left(\frac{1}{2}\varphi(2^{-n})\right)_n$, and the regularity properties of $\varphi$ ensure that 
\begin{equation}
\label{eqn:B_n}
\begin{split}
A_n^\mu&=\sup_{\theta\in\T}\langle M^\mu\rangle_n(\theta)\lesssim\left(\sum_{j=1}^n \varphi(2^{-n})^2\right)^\frac{1}{2}\\
&\simeq \left(\int_{2^{-n}}^1\frac{\varphi(t)^2}{t}\,dt\right)^\frac{1}{2}=\charphi(2^{-n}).
\end{split}
\end{equation}
This provides the discrete setting that we use to estimate the $p$- means of $1/|S_\mu|$. The link between the continuous and the discrete settings is given by the following Lemma. Given an interval $I\subseteq\T$ denote by $T_I$ the top-half of its Carleson box, that is, 
\[
T_I:=\{z\in\D\,|\, z/|z|\in I, |I|/2\le1-|z|^2<|I|\}.
\]
\begin{lemma}
\label{lemma:Poisson}
Let $\mu$ be a positive Borel measure on $\T$ such that the increments of the dyadic martingale $M^\mu$ are uniformly bounded. Then there exists a $C>0$ so that
\[
\frac{1}{|S_\mu(z)|}\lesssim e^{CM^\mu_{I_z}}\qquad z\in\D
\]
where $I_z$ is the unique dyadic interval such that $z\in T_{I_z}$.
\end{lemma}
\begin{proof}
Note that 
\[
\frac{1}{|S_\mu(z)|}=e^{P_\mu(z)},
\]
where $P_\mu$ is the Poisson integral of $\mu$. Let $z^*:=z/|z|$ and let $n\in\N$ be so that $|I_z|=2^{-n}$, that is, $2^{-(n+1)}<1-|z|^2\leq 2^{-n}$. Let $(I_m)_{m=-2^{n-1}-1}^{2^{n-1}-1}$ denote all the dyadic intervals of generation $n$, $I_0$ being $I_z$ and $\sup_{x \in I_m}
 \dist(I_0, x)=|m|2^{-n}$. Hence if $\theta\in I_m$, one has $|z^*-e^{i\theta}|\simeq |m|(1-|z|^2)$. Therefore
\[
\begin{split}
P_\mu(z)=&\int_\T\frac{1-|z|^2}{|1-e^{-i\theta}z|^2}\,d\mu(\theta)\\
\simeq&\int_\T\frac{1-|z|^2}{\max^2\{1-|z^2|, |z^*-e^{i\theta}|\}}\,d\mu(\theta)\\
=& M^\mu_{I_z}+\sum_{m}\frac{M^\mu_{I_m}}{m^2}\\
\leq& M^\mu_{I_z}+C+\sum_m\frac{|M^\mu_{I_z}-M^\mu_{I_m}|}{m^2}.
\end{split}
\]
Since $M^\mu$ has uniformly bounded increments, given any two dyadic intervals $I$ and $J$ of generation $n$ 
\[
|M^\mu_I-M^\mu_J|\lesssim\log_2|P(I, J)|,
\]
where $P(I, J)$ is the smallest common dyadic ancestor of $I$ and $J$. Therefore, by splitting the last sum according to the length of $P(I_z, I_m)$, 
\[
\sum_m\frac{|M^\mu_{I_z}-M^\mu_{I_m}|}{m^2}\lesssim\sum_{j=0}^{n-1}j\sum_{|m|=2^{j}}^{2^{j+1}}\frac{1}{m^2}\lesssim\sum_{j=0}^{n-1}\frac{j}{2^j}\leq C',
\]
concluding the proof.
\end{proof}
\begin{coro}
    \label{coro:means}
    For all $p>0$ there exists a $C_p>0$ such that, for any singular inner function $S_\mu$,
    \[
    \int_\T\frac{d\theta}{|S_\mu(re^{i\theta})|^p}\lesssim\charphi(1-r)e^{C_p\charphi(1-r)^2}, 
    \]
    provided that $\mu$ is $\varphi$-smooth and that $\varphi$ is a continuous, increasing function such that $\varphi(t)/t^\beta$ is almost decreasing for some $0<\beta<1$.
\end{coro}
\begin{proof}
Fix $r\in(0, 1)$ and let $n$ be so that $2^{-n-1}\le1-r^2<2^{-n}$. Thanks to Lemma \ref{lemma:Poisson}, one has
\[
\int_\T\frac{1}{|S_\mu(re^{i\theta})|^p}\,d\theta\lesssim\int_\T e^{C_p M^\mu_{n}(\theta)}\,d\theta\leq C_p A^\mu_ne^\frac{C_p^2\left(A^\mu_n\right)^2}{2},
\]
thanks to \eqref{eqn:conc_mart}. The desired estimate follows from \eqref{eqn:B_n}.
\end{proof}

\section{Multipliers of $D^p_{p-1}$}
\label{sec:multipliers}
Let $X$ be a Banach function space on a domain $\Omega\subset\C^d$. The multiplier algebra of $X$ is defined as 
\[
\Mult_X:=\{h\colon\Omega\to\C\,|\, hf\in X, \,\,f\in X\}.
\]
If point evaluations are bounded linear functionals in $X$, an application of the closed graph theorem yields that, for any $h$ in $\Mult_X$, the associated multiplication operator
\[
M_h\colon f\mapsto fh
\]
is bounded on $X$. This provides $\Mult_X$ with a Banach space structure, modulo defining $\|h\|_{\Mult_X}:=\|M_h\|_{B(X)}$. In a fairly general setting, a non-vanishing multiplier is cyclic if and only if its square is:
\begin{lemma}
\label{lemma:multipliers}
Assume that $X$ is a Banach space of holomorphic functions over $\D$. Assume that $X$ has a dense subspace formed by all analytic polynomials, and that each polynomial is a multiplier.
If $h \in \Mult_X$ and $h$ is cyclic in $X$ then $h^2$ is cyclic in $X$. 
\end{lemma}
\begin{proof}
Under our assumptions, the constant $1$ function is cyclic and thus, the fact that $h$ is cyclic is equivalent with the existence of a sequence of polynomials $\{p_n\}_{n\in \N}$ such that \[\|1-p_n h\|_X \rightarrow 0,\] as $n\rightarrow \infty$. Fix $\varepsilon >0$ and choose $n_0 \in \N$: \[\|1-p_n h\|_X \leq \frac{\varepsilon}{2},\] for all $n\geq n_0$. Define $q_m =p_n p_m$ where $m \geq n$ is to be determined later in terms of $n$. The triangle inequality gives \[\|1-q_m h^2\|_X  \leq\|1-p_n h\|_X  + \|p_n h (1-p_m h)\|_X.\]
The first summand of the right-hand side is bounded by $\frac{\varepsilon}{2}$ while the second term is bounded by \[\|p_n h\|_{\Mult_X} \|1-p_m h\|_X.\]
Let $C_n:=\|p_nh\|_{\Mult_X}$. It is enough to take $m$ large enough so that 
\[\|1-p_m h\|_X \leq \frac{\varepsilon}{2C_n},\] to obtain that 
\[\|1-q_m h^2\|_X \leq \varepsilon.\]
\end{proof}
\begin{coro}
\label{coro:cyclic_mult}
Let $h\colon\D\to\C$ such that $h^\beta$ is a multiplier of $X$ for all $\beta>0$. Then $h$ is cyclic provided that $h^\alpha$ is, for some $\alpha>0$.
\end{coro}
\begin{proof}
Pick $k$ in $\N$ so that $2^k\alpha>1$. By Lemma \ref{lemma:multipliers}, $h^{2^k\alpha}$ is cyclic in $X$. Since $h^{2^k\alpha}=h h^{2^k\alpha-1}\in[h]$, this yields that $h$ is cyclic as well.
\end{proof}

We now focus on the case $X=D^p_{p-1}$. In \cite{Wu99} it is shown that an analytic map $g$ on $\D$ is a multiplier of $D^p_{p-1}$ if and only if it is bounded and $$\nu_{g, p}:=|g'(z)|^p(1-|z|)^{p-1}$$ induces a Carleson measure for $D^p_{p-1}$, that is, it realizes a bounded embedding $D^p_{p-1}\subseteq L^p(\D, \nu_{g, p})$. It is interesting to note that Carleson measures for $D^p_{p-1}$ are characterized only for $p\le2$, and they are independent of $p$. Namely, a positive Borel measure $\nu$ is Carleson for $D^p_{p-1}$ if and only if there exists a constant $C$ such that, for all arcs $I\subseteq\D$,
\[
\nu(S(I))\leq C~|I|
\]
where 
\begin{equation}
\label{eqn:onebox}
S(I):=\{z\in\D\,|\, z/|z|\in I, 1-|z|\le|I| \}
\end{equation}
denotes the Carleson box associated to $I$. Condition \eqref{eqn:onebox} is usually referred to as the one-box condition. It is known that, for $p>2$, \eqref{eqn:onebox} is not sufficient for $\nu$ to be Carleson for $D^p_{p-1}$ \cite{Girela06JFA}. On the other hand, the stronger condition
\begin{equation}
\label{eqn:plog_onebox}
\nu(S(I))\lesssim |I|\left(\log\frac{e}{|I|}\right)^{1-\frac{p}{2}}\qquad I\subseteq\T
\end{equation}
is sufficient for $\mu$ to be Carleson for $D^p_{p-1}$, \cite{Pelaez14}. We will use the fact, proved in \cite[Lemma 4]{Anderson91}, that if $\mu$ is $\varphi$-smooth for some  positive, continuous, non decreasing function $\varphi$ on $(0, 1]$ such that $\varphi(t)/t^\beta$ is almost decreasing for some $0<\beta<1$, then 
\begin{equation}
\label{eqn:josechu_derivative}
\sup_{|z|=r}|S_\mu'(z)|\lesssim \frac{\varphi(1-r)}{1-r}.
\end{equation}
\begin{prop}
\label{prop:multiplier}
    Let $\mu$ be a singular measure on $\T$ satisfying \eqref{eqn:sufficient}. Then $S_\mu$ is a multiplier of $D^p_{p-1}$.
\end{prop}
\begin{proof}
By \eqref{eqn:josechu_derivative},  one sees that 
\[
|S_\mu'(z)|\lesssim \frac{1}{(1-r)\sqrt{\log\frac{e}{1-r}}}\qquad |z|=r.
\]
Hence, given an interval $I\subseteq\T$ and $p>2$
\[
\begin{split}
\nu_{S_\mu, p}(S(I))=&\int_{S(I)}|S_\mu'(z)|^p(1-|z|)^{p-1}~dA(z)\\
\lesssim&\int_{S(I)}\frac{dA(z)}{(1-|z|)\left(\log\frac{e}{1-|z|}\right)^\frac{p}{2}}\\
\simeq&|I|\int_{1-|I|}^1\frac{dr}{(1-r)\left(\log\frac{e}{1-r}\right)^\frac{p}{2}}\\
\simeq&|I|\left(\log\frac{e}{|I|}\right)^{1-\frac{p}{2}}.
\end{split}
\]
Thus $\nu_{S_\mu, p}$ satisfies \eqref{eqn:plog_onebox}, and $S_\mu$ is a multiplier of $D^p_{p-1}$.
\end{proof}

\section{Proof of Theorem \ref{theo:main}}
\label{sec:proof}

We are now ready for the proof of Theorem \ref{theo:main}. Let $\mu$ be a positive, singular Borel measure on $\T$ satisfying \eqref{eqn:sufficient}. By Proposition \ref{prop:multiplier}, $S_\mu^\beta$ is a multiplier of $D^p_{p-1}$ for all $\beta>0$. In view of Corollary \ref{coro:cyclic_mult}, it is enough to show the existence, for all $p>2$, of a power of $S_\mu$ that is cyclic in $D^p_{p-1}$. This is the case be for $\varphi$-smooth measures (according to Definition \ref{defi:smooth}), provided that $\varphi$ satisfies an integrability condition:
\begin{theo}
\label{theo:power}
Let $S_\mu$ be a singular inner function such that $\mu$ is $\varphi$-smooth. If 
\begin{equation}
\label{eqn:main_hyp}    
\int_0^1\frac{\varphi(t)^p}{t}\charphi(t)e^{\varepsilon\charphi(t)^2}\,dt<\infty
\end{equation}
for some $\varepsilon>0$, then there exists an $\alpha>0$ such that $S_\mu^\alpha$ is cyclic in $D^p_{p-1}$.
\end{theo}
\begin{rem}
\label{ex:sqrtlog}
Equation \eqref{eqn:sufficient} implies \eqref{eqn:main_hyp}. Indeed, if we take $\varphi(t)=\frac{C}{\sqrt{\log\frac{e}{t}}}$, then
\[\charphi(s)^2 := C \int_s^1 \frac{dt}{t \log \frac{e}{t}},\] which, for $s \in (0,1)$, is controlled from above by $\log\log \frac{e}{s}$. This means that
 
 \[
\charphi(t)\leq C\sqrt{\log\log\frac{e}{t}}, 
 \]
 and thus 
 \[
\int_0^1\frac{\varphi(t)^p}{t}\langle\varphi\rangle(t)e^{\varepsilon\langle\varphi\rangle(t)^2}\,dt\le C \int_0^1 \frac{\sqrt{\log\log\frac{e}{t}}}{t\left(\log\frac{e}{t}\right)^{\frac{p}{2}-C\varepsilon}}\,dt.
\]
Since $p>2$, one can choose $\varepsilon$ small enough so that $p/2-C\varepsilon>1$, so the latter integral is finite.
\end{rem}
 \begin{proof}[Proof of Theorem \ref{theo:power}]
  Thanks to \cite[Proposition 5]{BrownShields84} a function $f$ is cyclic in $D^p_{p-1}$, provided that 
\begin{equation}
\label{eqn:dilated}
\sup_{t\in(0, 1)}\left\|\frac{f}{f_t} \right\|_{D^p_{p-1}}<\infty,
\end{equation}
since for all $t\in(0, 1)$ the function $f_t(z):=f(tz)$ is analytic on an open neighborhood of $\overline{\D}$. 

Let $\alpha$ be a positive number to be fixed later. Thanks to \eqref{eqn:josechu_derivative}, we have
\begin{equation}
\label{eqn:Josechu_derivative_Smu}
\sup_{|z|=r}|(S^\alpha_\mu)'(z)|\leq \alpha C \frac{\varphi(1-r)}{(1-r)},
\end{equation}
for some positive $C$.
Let $g_t(z):=S^\alpha_\mu(z)/S^\alpha_\mu(tz)$. We wish to prove that $\sup_{t\in(0, 1)}\|g_t\|_{D^p_{p-1}}<\infty$. To this end, observe that since $S_\mu$ is inner we have 
\[
|g_t'(z)|\le\frac{|(S^\alpha_\mu)'(z)|}{|S_\mu(tz)|^\alpha}+t\frac{|(S^\alpha_\mu)'(tz)|}{|S_\mu(tz)|^{2\alpha}}.
\]
Therefore, thanks to Corollary \ref{coro:means} and \eqref{eqn:Josechu_derivative_Smu},
\[
\begin{split}
\|g_t\|_{D^p_{p-1}}^p=&\int_\D|g_t(z)|^p(1-|z|^2)^{p-1}\,dA(z)\\
\lesssim&\int_0^1\frac{\varphi(1-r)^p}{1-r}\int_\T\frac{1}{|S_\mu(tre^{i
\theta})|^{\alpha p}}\,d\theta\,dr \\+&\int_0^1\frac{\varphi(1-tr)^p}{1-tr}\int_\T\frac{1}{|S_\mu(tre^{i
\theta})|^{\alpha p}}\,d\theta\,dr\\
\lesssim&\int_0^1\frac{\varphi(1-r)^p}{1-r}\charphi(1-tr)e^{\alpha p C\charphi(1-tr)^2}\,dr \\ +&\int_0^1\frac{\varphi(1-tr)^p}{1-tr}\charphi(1-tr)e^{2\alpha p C\charphi(1-tr)^2}\,dr.
\end{split}
\]
By assumption, both $\varphi(s)/s$ and $\charphi$ are decreasing. Hence 
\[
\|g_t\|_{D^p_{p-1}}^p\lesssim\int_0^1\frac{\varphi(1-r)^p}{1-r}\charphi(1-r)e^{2\alpha p C\charphi(1-r)^2}\,dr, 
\]
which is finite (uniformly on $t$) provided that $0<\alpha\le\frac{\varepsilon}{2pC}$. 
\end{proof}
This concludes the proof of Theorem \ref{theo:main}.

\begin{rem}
\label{rem:multipliers}
Our argument shows that a singular inner function $S_\mu$ satisfying \eqref{eqn:sufficient} is a multiplier of $D^p_{p-1}$ but not a multiplier of $\ell^p_A$. By Proposition \ref{prop:multiplier}, $S_\mu$ is a multiplier of $D^p_{p-1}$. On the other hand, notice that if $S_\mu$ is in $\Mult_{\ell^p_A}$, then $S_\mu$ is in $\Mult_{\ell^q_A}$, $q$ being the dual exponent of $p$, since the two multiplier algebras coincide, \cite{Nikolski66}. In particular, $S_\mu$ belongs to $\ell^q_A$, since $\ell^q_A$ contains constant functions. This contradicts the fact that $S_\mu$ is cyclic in $\ell^p_A$: let $\phi_{S_\mu}$ be the functional defined on polynomials  by
\[
\phi_{S_\mu}(f):=\int_{0}^{2\pi} f(e^{i\theta}) e^{i\theta}\overline{S_\mu(e^{i\theta})}\,d\theta.
\]
If $S_\mu\in\ell_A^q$, then $\phi_{S_\mu}$ extends to a bounded linear functional acting on $\ell^p_A$. Since 
\[
\phi_{S_\mu}(z^mS_\mu)=\int_{0}^{2\pi}e^{i(m+1)\theta}\,d\theta=0
\]
for all $m=0, 1, \dots$, we found a non-zero bounded linear functional that annihilates $[S_\mu]_{\ell^p_A}$. Hence $S_\mu$ cannot be cyclic in $\ell^p_A$, giving the desired contradiction. 
\end{rem}

\section{Outer Functions in $D^p_{p-1}$}
\label{sec:outer}
Any outer function in $H^2$ is cyclic in $H^2$, and any outer function in $\B$ is (weak$^*$) cyclic in $\B$, \cite{Brown91}. Hence the following question naturally arises:
\begin{question}
\label{q:outer}
    Let $p>2$. Is any outer function in $D^p_{p-1}$ cyclic for $D^p_{p-1}$?
\end{question}
The aim of this Section is to address Question \ref{q:outer} under some extra assumptions on $f$. As a preliminary remark we note that if $p>2$ then $H^p\subset D^p_{p-1}$, \cite{Pavlovic}. Therefore any outer function in $H^p$ is cyclic in $D^p_{p-1}$. In what follows, we show that the same holds for outer functions in $D^p_{p-1}\cap\B$:

\begin{theo}
\label{theo:outerBloch}
    Let $f$ be an outer function in $D^p_{p-1}\cap\B$, $p>2$. Then $f$ is cyclic in $D^p_{p-1}$. 
\end{theo}
Since $D^p_{p-1}\cap\B\not\subset H^p$, Theorem \ref{theo:outerBloch} doesn't follow directly from the embedding $H^p\subset D^p_{p-1}$ mentioned above. The main tool we use is the following adaptation of the argument in \cite{Brown91}:

\begin{lemma}
    \label{lemma:int}
Let $p>1$, $\varphi$ in $D^p_{p-1}$ and $f$ in the Bloch space $\mathcal{B}$. Then 
\[
\sup_{0<t<1}\int_\D|(f(z)-f(tz))\varphi'(tz))|^p(1-|z|)^{p-1}dA(z)\lesssim_p\|\varphi\|_{D^p_{p-1}}^p\|f\|_{\mathcal{B}}^p
\]
\end{lemma}
\begin{proof}
Since $f(z)=\int_0^zf'(s)ds$, one gets
\[
|f(z)-f(tz)|\le\int_{tz}^z|f'(z)|d|z|\le\|f\|_\mathcal{B}\log\frac{1-t|z|}{1-|z|},
\]
where we used that $|f'(z)|\le\|f\|_\mathcal{B}/(1-|z|)$. Therefore, for all $t\in(0, 1)$,
\[
\begin{split}
&\int_\D|(f(z)-f(tz))\varphi'(tz)|^p(1-|z|)^{p-1}dA(z)\\
\le&\|f\|^p_\mathcal{B}\int_\D|\varphi'(tz)|^p|(1-t|z|)|^{p-1}\left(\frac{1-|z|}{1-t|z|}\right)^{p-1}\left(\log\frac{1-t|z|}{1-|z|}\right)^pdA(z).
\end{split}
\]
For all $t$ and $|z|$ in $(0, 1)$, $\frac{1-|z|}{1-t|z|}$ lies in $(0, 1)$. The Lemma then follows by observing that the function $x\mapsto x^{p-1}\left(\log\frac{1}{x}\right)^p$ is bounded on the unit interval, provided that $p>1$.
\end{proof}
As a Corollary, we obtain a sufficient condition for a function  in $D^p_{p-1}$ to be in the invariant subspace generated by another:
\begin{coro}
\label{coro:leq}
Let $p >1$, $f\in D^p_{p-1}\cap\mathcal{B}$ and $g$ in $D^p_{p-1}$ such that $g/f\in H^\infty$. Then $g\in [f]$.
\end{coro}
\begin{proof}
Let $\varphi:=g/f$. By assumption, $\varphi$ is a bounded analytic function. Since $\varphi_t f$ converges pointwise to $g$ and $\varphi_t f\in[f]$, it suffices to show that $\sup_t\|\varphi_tf\|_{D^p_{p-1}}<\infty$. To this end, notice that for all $0<t<1$
\[
\begin{split}
|(\varphi_tf)'|\le &|\varphi_t| |f'|+ |f-f_t|~|\varphi'_t|+|(\varphi_t)'f_t|\\
\le &\|\varphi\|_\infty (|f'|+|(f_t)'|)+ |f-f_t|~|\varphi'_t|+|(\varphi_tf_t)'|\\
:=& \|\varphi\|_\infty h_1+h_2+h_3,
\end{split}
\]
respectively. We wish to show that 
\begin{equation}
\label{eqn:int123}
\int_\D h_i^p(z)(1-|z|)^{p-1}dA(z)\qquad i=1, 2, 3
\end{equation}
are uniformly bounded on $t$. For $i=1$, the integral in \eqref{eqn:int123} is bounded since $f$ is in $D^p_{p-1}$ and dilates of $D^p_{p-1}$ functions converge in norm. Similarly, since $g=\varphi f$ is in $D^p_{p-1}$ the integral in \eqref{eqn:int123} is also uniformly bounded for $i=3$. For $i=2$, the integral in \eqref{eqn:int123} is bounded thanks to Lemma \ref{lemma:int}.
\end{proof} 

We are now ready for the proof of Theorem \ref{theo:outerBloch}.

\begin{proof}[Proof of Theorem \ref{theo:outerBloch}]
As in \cite[Th. 3]{Brown91}, we note that any bounded outer function is weak$^*$ cyclic in $H^\infty$, and therefore cyclic in $D^p_{p-1}$ since $H^\infty\subseteq D^p_{p-1}$ for all $p>2$ and the embedding is weak$^*$ continuous. Thanks to Corollary \ref{coro:leq}, we only have to show that there exists a bounded outer function $g$ in $D^p_{p-1}$ that is pointwise bounded by $f$. The function 
\[
g(z):=e^{\int_\T\frac{e^{i\theta}+z}{e^{i\theta}-z}~\log|g^*(e^{i\theta})|\,d\theta}\qquad z\in\D
\]
where
\[
g^*(e^{i\theta}):=\begin{cases}
1\qquad &\text{if}\,\,|f^*(e^{i\theta})|\ge1\\
|f^*(e^{i\theta})|&\text{otherwise}
\end{cases}
\]
is the function we seek.
\end{proof}

 \section{Logarithmic Conditions}
\label{sec:log}
As we pointed out in Section \ref{sec:martingales},  if $\nu$ is a $\varphi$-smooth measure on $\T$ and 
 \[
 \int_0^1\frac{\varphi^2(t)}{t}<\infty,
 \]
 then $\nu$ is absolutely continuous with respect to Lebesgue measure. On the other hand, if $\int_0^1\varphi(t)^2/t=\infty$, then there exists a $\varphi$-smooth singular inner measure (see \cite{Kahane69}). In particular, there exists a singular $\varphi$-smooth measure on $\T$ 
 such that 
 \begin{equation}
 \label{eqn:varphi_p}
 \int_0^1\frac{\varphi^p(t)}{t}<\infty,
 \end{equation}
 for all $p>2$. It turns out that singular inner functions generated by such measures have logarithms in $D^p_{p-1}$:
\begin{lemma}
    \label{prop:psum_nu}
Let $p>0$ and let 
\begin{equation}
\label{eqn:f_nu}
f_\nu(z):=\exp\left\{-\int_\T\frac{w+z}{w-z}\,d\nu(w)\right\},
\end{equation}
where $\nu$ is a positive Borel $\varphi$-smooth measure. Assume that $\varphi$ satisfies the additional integrability condition \eqref{eqn:varphi_p}. Then $\log f_\nu$ is in $D^p_{p-1}$.
\end{lemma}
\begin{proof}
Thanks to \cite[Lemma 4]{Anderson91} one obtains that $|(\log f_\nu)'|\lesssim\frac{\varphi(1-r)}{1-r}$, hence
\[
\int_{\D}|\left(\log f_\nu\right)'(z)|^p(1-|z|)^{p-1}dA(z)\lesssim \int_0^1\frac{\varphi(t)^p}{t}\,dt<\infty.
\]
\end{proof}

When $p \ge 2$, Lemma \ref{prop:psum_nu} and the embedding into $\ell^p_A$ spaces yield a summability property for the Fourier coefficients of $\nu$
\[
\hat{\nu}(z):=\int_\T e^{-in\theta}d\nu(\theta)\qquad n\in\Z
\]
provided that $\nu$ is sufficiently smooth.
\begin{coro}
\label{coro:Fourier_ellp}
Let $p\ge2$ and let $\nu$ be a positive singular $\varphi$-smooth measure such that $\varphi$ satisfying the $p$ integrability condition \eqref{eqn:varphi_p}. Then 
\[
\sum_{n\in\Z}|\hat\nu(n)|^p<\infty.
\]
\end{coro}

In particular, if $p=2$ then the function $f_\nu$  must be outer.

\begin{proof}
Since $\nu$ is a positive measure, $\hat\nu(-n)=\overline{\hat\nu(n)}$ for all $n\in\Z$. Hence it suffices to show that $(\hat\nu(n))_{n\ge0}$ is in $\ell^p$. Let $f_\nu$ be defined as in \eqref{eqn:f_nu}. Thanks to Proposition \ref{prop:psum_nu}, $\log f_\nu$ is in $D^p_{p-1}$. By the embedding \eqref{eqn:emb_ge2}, 
\begin{equation}
\label{eqn:log_ellp}
\log f_\nu=1+2\sum_{n\ge1}\hat\nu(n)z^n
\end{equation}
is in $\ell^p_A$. 
\end{proof}

Notice that \eqref{eqn:varphi_p} is a weaker condition than the main hypothesis of Theorem \ref{theo:power}. The following question arises then naturally:

\begin{question}
\label{q:Besov}
    Let $f$ be an analytic function on $\D$ such that both $f$ and $\log f$ belong to $D^p_{p-1}$, $p>2$. Is $f$ cyclic in $D^p_{p-1}$?
\end{question}

The cyclicity of functions whose logarithm belongs to a given space has been studied in  various settings. For instance, the analog of Question \ref{q:Besov} has a positive answer when replacing $D^p_{p-1}$ with the Dirichlet space \cite{BCL+15, Aleman} or with Hardy spaces $H^p$ as a direct consequence of Theorem II.4.6 and Theorem II.7.4 in \cite{garnett}. For certain Bergman-type spaces this is evidently false: let $\alpha <0$ and consider the spaces $\ell^{2, \alpha}_A$ defined in \eqref{eqn:ell2alpha}. Recall that, in this setting, a singular inner function is cyclic exactly when the corresponding measure places no mass on any Beurling-Carleson set. This condition is not satisfied by the atomic singular inner function \[f(z) = e^{-\frac{1+z}{1-z}}.\] However, its logarithm is in all $\ell^{2,\alpha}_A$ spaces when $\alpha <-1$. The conditions studied in \cite{ivriinicolau} for membership of singular inner functions in Dirichlet-type spaces probably provide counterexamples to Question 3 for other Bergman spaces with values of $\alpha$ closer to 0: if a singular inner function $S_\mu$ belongs to $\ell^{2,\alpha}_A$ and to its Cauchy dual $\ell^{2,-\alpha}_A$, then $S_\mu$ can't be cyclic in either of these spaces since its backward shift $BS_\mu$ does not belong to the invariant subspace generated by $S_\mu$. 

To the best of our knowledge, Question \ref{q:Besov} remains open even for singular inner functions. On the other hand, the analog question for $\ell^p_A$ has a negative answer. 

\begin{theo}
\label{theo:log_notcyclic}
For all $p>2$, there exists a positive singular measure $\mu$ on $\T$ which is supported on a Beurling-Carleson set whose Fourier coefficients are in $\ell^p$. 
\end{theo}
In particular, since $(\hat{\mu}(n))_{n\in\Z}$ is in $\ell^p$, then $\log S_\mu$ is in $\ell^p_A$ by \eqref{eqn:log_ellp}. On the other hand, $S_\mu$ is not cyclic in $\ell^p_A$ by Theorem \ref{theo:necessity}.
\begin{proof}[Proof of Theorem \ref{theo:log_notcyclic}]
Our claim follows by a construction of Salem, \cite[Th. II]{Salem51}, where it is shown that for all $\alpha\in (0, 1)$ and $\varepsilon>0$, there exists a positive singular measure $\mu$ supported on a perfect set $E$ of Hausdorff dimension $\alpha$ such that 
\begin{equation}
\label{eqn:fourier_bigo}
\hat{\mu}(n)=\mathcal{O}(1/n^{\frac{\alpha}{2}-\varepsilon}).
\end{equation}
Fixed $p>2$, for the Fourier coefficient of $\mu$ to lie in $\ell^p$, it suffices to choose $\alpha$ and $\varepsilon$ in the allowed range so that $p\left(\frac{\alpha}{2}-\varepsilon\right)>1$.

To conclude, we need to observe that a more detailed inspection of Salem's proof shows that $E$ is a Beurling-Carleson set. Indeed, $E$ is a Cantor-type set constructed as follows: two parameters $d \in \N$, $d\geq 2$, and $\xi$ so that $0<\xi<1/d$ are given, and they depend exclusively on $\alpha$ and $\varepsilon$. Then $E$ is the Cantor set constructed recursively as follows. $E_0=\T$ and for all interval $I$ of generation $j-1$ composing $E_{j-1}$ we pick $d$ sub-intervals whose length is equal to $\xi_j|I|$ and such that two consecutive left endpoints are at distance $\nu|I|$, where $\nu:=\frac{\xi+1/d}{2}$. The union of such sub-intervals composes $E_j$ and the set $E$ is defined as $\bigcap_jE_j$. Even if the sequence $(\xi_j)_j$ is not given explicitly, it is shown that there exists one choice for such sequence satisfying
\begin{equation}
\label{eqn:xi_n}
\left(1-\frac{1}{(j+1)^2}\right)\xi\le\xi_j\le\xi
\end{equation}
and a positive measure $\mu$ supported on $E$ satisfying \eqref{eqn:fourier_bigo} (in fact, this is the case for a generic sequence $(\xi_n)$ satisfying \eqref{eqn:xi_n}). Therefore, we only need to show that any possible choice satisfying \eqref{eqn:xi_n} provides a set $E$ that is a Beurling-Carleson set. At step $j$, the construction of the set $E$ discards $d^j$ intervals of length $\Gamma_j$, where, from \eqref{eqn:xi_n} we obtain
\[
\Gamma_j=\prod_{i=1}^j(\nu-\xi_i)\le \prod_{i=1}^j\left(\frac{\xi+1/d}{2}-\xi\left(1-\frac{1}{(i+1)^2}\right)\right).\]
After canceling the $\xi$ terms and taking common factor, this gives 
\[\Gamma_j \leq \left(\frac{1/d-\xi}{2}\right)^je^{\sum_{i=1}^j\log \left(1+\frac{\frac{2\xi}{1/d-\xi}}{(i+1)^2}\right)}.\]
Then, we can apply standard estimates on the logarithm function to arrive to
\[\Gamma_j \le \frac{1}{(2d)^j}e^{\left(\frac{\pi^2}{6}-1\right)\cdot \frac{2\xi}{1/d-\xi}} \le \frac{C_{\alpha, \varepsilon}}{(2d)^j}.
\]
Therefore $\sum_jd^j\Gamma_j\log\frac{1}{\Gamma_j}<\infty$, and $E$ is a Beurling-Carleson set thanks to \eqref{eqn:bc}. 

\end{proof}

\bibliographystyle{amsplain} 
\bibliography{bibliography}

\end{document}